\documentclass{amsart}

\newtheorem*{theorem}{Theorem}

\newtheorem{corollary}{Corollary}

\theoremstyle{definition}

\theoremstyle{remark}

\numberwithin{equation}{section}

\begin{document}

\title{Hardy-Littlewood inequality for primes}

\author{V.V. Miasoyedov}
\address{Ponomareva street, 2b, flat 52, Irpen', Kotsubinskoe, 08298, Kiev region, Ukraine}
\email{generalist@mail.ru}

\thanks{This work was completed privately}

\begin{abstract}
In the article we establish the Hardy-Littlewood inequality $ \pi (x + y) \leq
\pi (x) + \pi (y) $. We also prove that the naturally ordered primes
$p_1=2,p_2=3,p_3=5,p_4=7,\dots$ satisfy the
inequality $ p_ {a + b}>  p_a + p_b $ for all $a, \ b \geq 2$.

\medskip


\medskip


\end{abstract}

\maketitle

\section{Introduction}

The conjecture that the distribution of primes satisfies the inequality
 \begin {equation}
 \pi (x + y) \leq \pi (x) + \pi (y)
 \end {equation}
 for all $ x, \ y \geq 2 $, was formulated by Hardy and Littlewood in connection with the weak Goldbach problem [1, 2]. This conjecture has been subjected to intensive study by numerical methods [3] along with another conjecture of these authors [1], which contradicts the former. The article by G. Mincu [4] contains some results achievable by methods used in analytic number theory and computer calculations in the initial range of numbers.

\section{The proof of inequality}
\begin {theorem} [Hardy-Littlewood, 1923] The distribution of prime numbers satisfies the inequality
$$ \pi (x+y) \leq \pi (x) + \pi (y) $$
for all $ x, \ y \geq 2 $.
\end {theorem}
\begin{proof} To prove (1.1) we apply transfinite induction. Inequality (1.1) is satisfied when $x+y=4$, $x=2$, and $y=2$, or $x+y=5$, $x=2$, and $y=3$. Let it be satisfied for all $x+y<N$ and prove it for $x+y=N>5$.

 Suppose the contrary that
$$ \pi (x + y) >\pi (x) + \pi (y)$$
 for some pair $x$ and $y$ with $x+y=N$.

 Let $ \pi (z) = \pi (x) + \pi (y) $ for some $z>3$ and choose $z=p-1$, where $p$ is prime. Then we have the inequality $ \pi (x + y)>\pi (z)  $ and, by the monotonicity of $ \pi (x) $, also the inequality $ x + y>z $. Therefore we have $x>z-y$ and $y>z-x$.

 Both the numbers $ z-y $ and $ z-x $ are greater than one. Indeed, if for example $z-y=1$, then $\pi(y+1)=\pi(x)+\pi(y)$ from which it follows that $x=2$, and $z=y+1$ is prime. This is impossible by choice of $z$. Then, particularly, both the numbers $x$ and $y$ are greater than $2$ because, for example, $x>z-y\geq 2$.

So we have $ \pi (y) \geq \pi (z-x)$. We increase both the sides of this inequality by $\pi(x)$ to obtain
\begin {equation}
\pi (z)\geq\pi(x)+\pi (z-x).
\end {equation}
If the inequality (2.1) is strict, then the sum of integers $ x'= x $ $ y'= z-x $ is less than $ N $, and inequality (1.1) does not hold. This contradicts to the assumption of induction.

If the inequality (2.1) is equality, that is when $ \pi (y)=\pi (z-x)$,
\begin {equation}
\pi (z)=\pi(x)+\pi (z-x),
 \end {equation}
then  we rewrite (2.2) by our choice $z=p-1$ as
$$\pi (p)-1=\pi(x)+\pi (p-x),$$
where $\pi(p-x)=\pi(z-x)$. This is because $p-x=z+1-x\leq y$.
From this follows the inequality
$$
\pi (p)>\pi(x)+\pi (p-x).
$$
If $ p<x + y $, then we have found a pair of integers $ x'=x $, $y' =p-x $  that does not satisfy inequality (1.1) and has the sum $ p<N $. Again, we have obtained a contradiction to the assumption of induction.

If $ p = x + y=N $ then  by definition of number $z$ within equality $\pi(z)=\pi(x)+\pi(y)$
 \begin {equation}
 \pi(x+y-1)=\pi (x) + \pi (y).
 \end {equation}


Let solutions of equation (2.3) exist, otherwise we have nothing more to prove. We can consequently vary numbers $x$ and $y$ saving their sum and the  values of items of equation. We rewrite equation (2.3) in the form
\begin{equation}
 \pi(x+y-1)=\pi (x-v_x) + \pi (y+v_y),
 \end{equation}
where $v_x$ and $v_y$, $v_x=v_y$, can be chosen depending on circumstances and save the values of items of equation (2.3). Those circumstances are
the closest intervals with prime borders $p_x\leq x<p'_x$, $p_y\leq y<p'_y$ and $p_z\leq x+y-1<N$ which contain some solution. Note that if $p_z=p_x$ then from the equation it follows that $\pi(y)=0$ which is impossible.

If $x-p_x\leq p'_y-1-y$ we take $v_y=v_x=x-p_x$ and obtain equation $\pi(p-1)=\pi(p_x)+\pi(p-p_x)$. Hereof we have inequality
$$\pi(p-1)>\pi(p_x-1)+\pi(p-p_x).$$
We see that the pair of numbers $p_x-1$ and $p-p_x$ does not satisfy inequality (1.1) and has the sum $p-1<N$. This contradicts to the assumption of induction.

 Otherwise $x-p_x>p'_y-1-y$, and we take $v_x=v_y=p'_y-1-y$ in (2.4) to obtain the equation
 \begin{equation}
 \pi(p-1)=\pi(p-p'_y+1)+\pi(p'_y-1),
 \end{equation}
 where $p-p'_y+1>p_x$.
 Let solutions of equation (2.5) exist, otherwise solutions of equation (2.3) do not exist also because values of items of equation (2.5) are the same in equation (2.3). As above, change the form of equation to
 \begin{equation}
 \pi(p-1-v_z)=\pi(p-p'_y+1-v_x)+\pi(p'_y-1),
 \end{equation}
 where $v_z$ and $v_x$ save the values of items of equation (2.5).

 If $p-p'_y+1-p_x\leq p-1-p_z$ we take $v_z=v_x=p-p'_y+1-p_x$ and obtain the equation $\pi(p'_y+p_x-2)=\pi(p_x)+\pi(p'_y-1)$. Hence we have inequality
 $$\pi(p'_y+p_x-2)>\pi(p_x-1)+\pi(p'_y-1).$$
We have found the pair of numbers $p_x-1$ and $p'_y-1$ which does not satisfy inequality (1.1) and whose sum $p'_y+p_x-2$ is less than $p'_y+p_x-1<N$ by condition which implies equation (2.5). This contradicts to the assumption of induction.

Otherwise $p-p'_y+1-p_x>p-1-p_z$, and we take $v_x=v_z=p-1-p_z$ in (2.6) to obtain the equation
\begin{equation}
\pi(p_z)=\pi(p_z-p'_y+2)+\pi(p'_y-1).
\end{equation}
Let solutions of equation (2.7) exist, otherwise solutions of equation (2.5) do not exist also because values of items of equation (2.7) are the same in equation (2.5).

Now we rearrange equation (2.7) to the form
\begin{equation}
\pi(p_z-1)=\pi(p_z-p'_y+2)+\pi(p'_y-1)-1.
\end{equation}
We have to use adjacent intervals. So let $\widehat{p}_z$, $\widehat{p}_x$ and $\widehat{p}_y$ be the primes closest to the left to $p_z$, $p_x$ and $p_y$ respectively. 

Let first that the item $-1$ in (2.8) is related to the item $\pi(p'_y-1)$. Then there exist independent variations $v_y$ and $v_x$ such that $\pi(p'_y-1-v_y)=\pi(p'_y-1)-1$ and $\pi(p_z-p'_y+2+v_x)=\pi(p_z-p'_y+2)$. We rewrite equation (2.8) in form
\begin{equation}
\pi(p_z-1)=\pi(p_z-p'_y+2+v_x)+\pi(p'_y-1-v_y).
\end{equation}
If $p'_y-1-\widehat{p}_y\leq p'_x-1-p_z+p'_y-2$, then we take $v_x=v_y=p'_y-1-\widehat{p}_y$ and obtain the equation $\pi(p_z-1)=\pi(p_z+1-\widehat{p}_y)+\pi(\widehat{p}_y)$. From this it follows inequality
$$\pi(p_z-1)>\pi(p_z+1-\widehat{p}_y)+\pi(\widehat{p}_y-2).$$
The pair of numbers $p_z+1-\widehat{p}_y$ and $\widehat{p}_y-2$ does not satisfy inequality (1.1) and has the sum $p_z-1<N$. This contradicts to the assumption of induction.

Otherwise $p'_y-1-\widehat{p}_y>p'_x-1-p_z+p'_y-2$, and we take $v_y=v_x=p'_x-1-p_z+p'_y-2$ in (2.9) to obtain the equation
\begin{equation}
\pi(p_z-1)=\pi(p'_x-1)+\pi(p_z-p'_x+2),
\end{equation}
where $p_z-p'_x+2<p_y$ by condition $\pi(p'_y-1-v_y)=\pi(p'_y-1)-1$.
Let solutions of equation (2.10) exist, otherwise $\pi(p_z-1)\ne\pi(p'_x-1)+\pi(p_z-p'_x)$, and we use the variations $v'_y=v'_x=-v_y$ of arguments of the items  in the inequality  $\pi(p_z-1)\ne\pi(p'_x-1+v'_x)+\pi(p_z-p'_x-v'_y)$ to conclude that solutions of equation (2.8) would not exist also.

In this case starting from equation (2.4) with inequality $y-p_y\leq p'_x-1-x$,  we find pairs of numbers whose existence contradicts to the assumption of induction, or obtain the sequence of equations and conditions
$$ \pi(p-1)=\pi(p'_x-1)+\pi(p-p'_x+1), \ p>p'_x-1+p_y ,$$
\begin{equation}\pi(p_z)=\pi(p'_x-1)+\pi(p_z-p'_x+2), \ p_z>p'_x-1+p_y,\end{equation}
along with the chain of assumptions of existence of their solutions till equation (2.3).
But conditions of equations (2.11) and (2.10) are incompatible for number $p_z-p'_x+1$, because we have $p_z>p'_x-1+p_y$ and $p_z-p'_x+2<p_y$ simultaneously, that is $p_y<p_z-p'_x+1<p_y-1$. This is impossible. So equation (2.3) cannot have solutions in this case, or by symmetric consideration we have found some pair whose existence contradicts to the assumption of induction\footnotemark\footnotetext{Note that if we begin consideration by symmetry of $x$ and $y$ then an equation similar to equation (2.10) will be met \textit{later} than the similar chain of equations and conditions which in text follow equation (2.10). So we meet a contradiction by position in this case also.}.

Thus if item $-1$ relates to the item $\pi(p'_y-1)$ we have found the pair of numbers whose existence contradicts to the assumption of induction, or meet impossible condition.

Now let that item $-1$ in (2.8) relates to the item $\pi(p_z-p'_y+2)$. Then there exist independent variations $v_z$ and $v_x$, such that $\pi(p_z-1-v_z)=\pi(p_z-1)$ and $\pi(p_z-p'_y+2-v_x)=\pi(p_z-p'_y+2)-1$. We rewrite equation (2.8) as
\begin{equation}
\pi(p_z-1-v_z)=\pi(p_z-p'_y+1-v_x)+\pi(p'_y-1).
\end{equation}
If $p_z-p'_y+1-\widehat{p}_x\leq p_z-1-\widehat{p}_z$, then we take $v_z=v_x=p_z-p'_y+1-\widehat{p}_x$ and obtain the equation $\pi(p'_y+\widehat{p}_x-2)=\pi(\widehat{p}_x)+\pi(p'_y-1)$. Hereof we have the inequality
$$\pi(p'_y+\widehat{p}_x-2)>\pi(\widehat{p}_x-1)+\pi(p'_y-1).$$
This inequality contains the pair of numbers $\widehat{p}_x-1$ and $p'_y-1$ which does not satisfy inequality (1.1) and has the sum $p'_y+\widehat{p}_x-2<p_z$ which is less than $N$. This contradicts to the assumption of induction.

Otherwise $p_z-p'_y+1-\widehat{p}_x>p_z-1-\widehat{p}_z$, and we take in (2.12) $v_x=v_z=p_z-1-\widehat{p}_z$ to obtain the equation
$$\pi(\widehat{p}_z)=\pi(\widehat{p}_z-p'_y+2)+\pi(p'_y-1),$$
where $\widehat{p}_z-p'_y+1<p_x$ by condition $\pi(p_z-p'_y+2-v_x)=\pi(p_z-p'_y+2)-1$.
This equation reproduces equation (2.7) with the smaller borders $\widehat{p}_z$ and $\widehat{p}_x$ while the border $p'_y-1$ is unchanged.

We can repeat above consideration of equation (2.7) with smaller pairs of borders $\widehat{p}_z$ and $\widehat{p}_x$. Particularly we reuse unchanged condition of equation (2.11).

The set of borders $p_x, \widehat{p}_x,\dots$ has the minimal element. So the process of infinite descending of borders will be finished in a finite number of steps. On each step we find pairs of numbers whose existence contradicts to the assumption of induction or we get an equation which does not have solutions.

Thus in all cases we have a contradiction to the assumption of induction. Hence the theorem is proved.
\end{proof}

From the proof of theorem it follows
\begin{corollary} Equation
$$\pi(p-1)=\pi(x)+\pi(p-x)$$
does not hold for all $x\geq 2$.
\end{corollary}
\begin{proof}
Really, we can omit all the cases when we could find the pairs which contradicts to the statement of the proved theorem. Otherwise after infinite descending we reach the equation which does not have solutions. Then we can fit variations of arguments of function $\pi(x)$ in items of any equation in the chain till equation (2.3), which is the equation of corollary, to conclude that all equations in the chain does not have solutions. All variations which does not contradict to the conditions of equations formulations are excepted because the values of items remain the same as for  variations used to fit process. Hence corollary is proved.
\end{proof}

An immediate consequence of the theorem (see also [5], [6]) is the following
\begin {corollary} The naturally ordered prime numbers $p_1=2,p_2=3,p_3=5,p_4=7,\dots$  satisfy the inequality
\begin {equation}
p_{a+b}>p_a + p_b,
\end {equation}
for all $ a, b \geq 2 $.
 \end {corollary}

\begin{proof}
We use the indices of the primes $ a, b \geq 2 $ to obtain  $ a + b = \pi (p_ {a + b}) = \pi (p_ {a}) + \pi (p_ {b}) $. By the theorem above, we have $ \pi (p_ {a + b}) = \pi (p_ {a}) + \pi (p_ {b}) \geq \pi (p_a + p_b) $. We eliminate the possibility of equality of $ \pi (p_ {a}) + \pi (p_ {b}) = \pi (p_a + p_b) $, otherwise decreasing $p_a$ and $p_a+p_b$ by $1$ we would have a contradiction with the proved theorem. Then $ \pi (p_ {a + b})> \pi (p_a + p_b) $, and by the monotonicity of $ \pi (x) $ we obtain the required inequality $ p_a + p_b <p_ {a + b} $. Hence the corollary is proved.
\end{proof}

\section{Conclusion}
We have considered the specific case of functional equation for primes counter by the way. In the another case equation $\pi(p)=\pi(x)+\pi(p-x)$ admits, by computations, the three sets of solutions $\{2\}$, $\{2,3,4\}$, and $\{2,3,4,9,10\}$ depending on $p$.

Inequality (2.13) may include an arbitrary number of items, the same in both sums. Some strange restricted periodicity was noticed after computations of exponent $4/3$ for deviations  from the trend $4L^2$, $L\in [1,1436]$ of differences $$\sum {p_i}-p_ {\sum {i}}$$ with $L$ items in sums. Probably, such periodicity possesses recursive features for deviations from further trends. Nevertheless this phenomena requires more advanced calculations and analytical skills.

In particular, well defined subsets of prime numbers can be included to sums. For such subsets we could literally reprove the theorem above to obtain inequality $\pi_f(x)+\pi_f(y)\geq \pi_f(x+y)-A_f$. For example, in relation to bounded gap problem we could calculate the sequence of $A_g$.

An interesting conjecture is that
$$3\pi(x-1)\leq 2\pi(2x-1)$$
is valid for all natural $x$.

\bibliographystyle{amsplain}

\end{document}